\documentclass[a4paper,reqno]{amsart}
\usepackage{amssymb, amsmath, amscd}

\def\grad{\operatorname{grad}}
\def\XIx{{h}}
\newtheorem{theorem}{Theorem}[section]
\newtheorem{definition}[theorem]{Definition}
\newtheorem{lemma}[theorem]{Lemma}
\newtheorem{proposition}[theorem]{Proposition}
\newtheorem{remark}[theorem]{Remark}
\newtheorem{corollary}[theorem]{Corollary}

\def\wedgeo{\wedge}
\def\mapright#1{\smash{\mathop{\longrightarrow}\limits\sp{#1}}}
\makeatletter
 \@addtoreset{equation}{section}
\makeatother
\begin{document}
\title{Curvature Properties of  Weyl Geometries}
\author{Peter  Gilkey, Stana Nik\v cevi\'c, and Udo Simon}
\address{PG: Mathematics Department, \; University of Oregon, \;\;
  Eugene \; OR 97403 \; USA\\
  E-mail: gilkey@uoregon.edu}
\address{SN: Mathematical Institute, Sanu,
Knez Mihailova 36, p.p. 367,
11001 Belgrade,
Serbia.\\ Email: stanan@mi.sanu.ac.rs}
\address{US: Institut f\"ur Mathematik, Technische Universit\"at
Berlin 
\newline  Strasse des 17. Juni 135, D-10623 Berlin, Germany\\ Email: simon@math.tu-berlin.de}
\begin{abstract}{We examine relations between geometry and the associated\\
curvature decompositions in Weyl  Geometry.}\end{abstract}
\maketitle
\centerline{\bf Dedicated to Heinrich Wefelscheid}\centerline{\bf on the occasion of his 70th birthday}

\section{\bf Introduction}\label{sect-1}

\subsection{Weyl geometry} Let $N$ be a smooth manifold of dimension $n\ge3$. Let $\nabla$ be a torsion free
connection on the tangent bundle $TN$ of $N$ and let $g$ be a semi-Riemannian metric on $N$. Then the triple
$\mathcal{W}:=(N,g,\nabla)$ is said to be a {\it Weyl manifold} if there exists a smooth $1$-form $\phi\in C^\infty(T^*N)$ so
that:
\begin{equation}\label{eqn-1.a}
\nabla g=-2\phi\otimes g\,.
\end{equation}
Weyl geometry \cite {W22} is linked with conformal geometry. If $f\in C^\infty(N)$,
let
$g_1:=e^{2f}g$ be a conformally equivalent metric. If
$\mathcal{W}=(N,g,\nabla)$ is a Weyl manifold, then
the triple $(N, g_1,\nabla)$ is again a Weyl manifold where the associated
$1$-form is given by taking $\phi_1:=\phi - df$. The transformation of the pair
$(g,\phi)\rightarrow(g_1,\phi_1)$ is called a {\it gauge transformation}. Properties  of the Weyl geometry that
are invariant under gauge transformations are called {\it gauge invariants}.

Let $\nabla^{g}$ be the Levi-Civita connection of $g$. There exists a
conformally equivalent metric
$g_1$ locally so that $\nabla=\nabla^{g_1}$ if and only if $d\phi=0$; if $d\phi=0$, such a class exists
globally if and only if the associated de Rham cohomology class $[\phi]$ vanishes.

\subsection{Affine and Riemannian geometry} We say that the pair
$\mathcal{A}:=(N,\nabla)$ is an {\it affine manifold} if
$\nabla$ is a torsion free connection on $TN$. Similarly, we say that the pair $\mathcal{N}:=(N,g)$ is a {\it
semi-Riemannian manifold} if $g$ is a semi-Riemannian metric on $N$. Weyl geometry lies
between affine geometry and semi-Riemannian geometry. Every Weyl manifold gives rise both to  an underlying
affine manifold $(N,\nabla)$ and to an underlying semi-Riemannian manifold $(N,g)$; Equation
(\ref{eqn-1.a}) provides the link between these two structures.
Since the Levi-Civita connection $\nabla^{g}$ is torsion free and since $\nabla^g g=0$, the triple $(N,g,\nabla^{g})$ is a Weyl
manifold. There are, however, examples with $d\phi\ne0$, so Weyl geometry is more general than
semi-Riemannian geometry or even than conformal semi-Riemannian geometry. 

\subsection{Curvature} The {\it curvature operator} $\mathcal{R}$ of a torsion free
connection $\nabla$ is the element of
$\otimes^2T^*N\otimes\operatorname{End}(TN)$ which is defined by:
$$\mathcal{R}(x,y)z := (\nabla_x\nabla_y-\nabla_y\nabla_x-\nabla_{[x,y]})z\,.$$
The $g$-associated {\it curvature tensor} is given by using the metric to lower an index:
$$R(x,y,z,w):=g(\mathcal{R}(x,y)z,w)\,.$$
We have the following identities: 
\begin{eqnarray}
&&R(x,y,z,w)+R(y,x,z,w)=0,\quad\text{and}\label{eqn-1.b}\\
&&R(x,y,z,w)+R(y,z,x,w)+R(z,x,y,w)=0\label{eqn-1.c}\,.
\end{eqnarray}
The relation of Equation (\ref{eqn-1.c}) is called the {\it Bianchi identity}. 
The {\it Ricci tensor}
$\operatorname{Ric} := \operatorname{Ric}(R)  := \operatorname{Ric}(\mathcal{R})$ is defined by setting:
$$
\operatorname{Ric}(x,y):=\operatorname{Tr}\{z\rightarrow (\mathcal{R}(z,x)y\} \,.
$$
The tensor $\operatorname{Ric}(\mathcal{R})$ does not depend on the metric $g$ and is a gauge invariant. Let $g_{ij}:=g(e_i,e_j)$
give the components of the metric tensor $g$ relative to a local frame $\{e_i\}$ for $TN$. Let $g^{ij}$ be the components of the
inverse matrix
$g^{-1}$  relative to the dual frame $\{e^i\}$ on $T^*N$.  We adopt the {\it Einstein convention} and sum over repeated indices.
We then express:
$$\operatorname{Ric}(x,y) =g^{ij}R(e_i,x,y,e_j)\,.$$
We contract again to define the {\it scalar curvature} of $R$ with respect to $g$
by setting:
$$\tau_g:=\operatorname{Tr}_g\operatorname{Ric}:=g^{ij}\operatorname{Ric}(e_i,e_j)\,.$$
We can define another Ricci-type tensor $\operatorname{Ric}^\star:= \operatorname{Ric}^\star(R)$  by
setting: 
$$
\operatorname{Ric}^\star(x,y):=g^{ij}R(x,e_i,e_j,y).
$$
We let $R_{ij}$ and $R^\star_{ij}$ be the components of these two tensors:
$$R_{ij} := \operatorname{Ric}(e_i,e_j)\quad\text{and}\quad R^\star_{ij} := \operatorname{Ric}^\star(e_i,e_j)\,.$$

We may decompose any (0,2)-tensor $\theta$ in the form $\theta=S\theta+\Lambda\theta$ where $S\theta$ and $\Lambda\theta$ are
the symmetrization and  the anti-symmetrization, respectively, of $\theta$. We have the following result (see, for example, the
discussion in
\cite{GNU10,W22}):
\begin{theorem}\label{thm-1.1}
Let $(N,g,\nabla)$ be  a Weyl manifold. Let $R=R(\nabla)$. Then we have
\begin{eqnarray}
&&R(x,y,z,w)+R(x,y,w,z)=\textstyle\frac2n\{\operatorname{Ric}(y,x) - \operatorname{Ric}(x,y)\}g(z,w) \nonumber\\
&&\qquad= - \tfrac4n\, (\Lambda\operatorname{Ric})(x,y)\,g(z,w),\label{eqn-1.e}\\
&&S\operatorname{Ric}^\star = S\operatorname{Ric},\quad   \Lambda
\operatorname{Ric}^\star=\textstyle\frac{n-4}n\Lambda \operatorname{Ric},\quad\text{and}\quad
\tau_g =\operatorname{Tr}_g\operatorname{Ric}^\star\,.
\end{eqnarray}
\end{theorem}
\begin{remark}\rm
If $\nabla=\nabla^{g}$ arises from semi-Riemannian geometry, then one has an additional symmetry:
\begin{equation}\label{eqn-1.f}
R(x,y,z,w) + R(x,y,w,z)=0 \,.
\end{equation}
\end{remark}
\subsection{The algebraic context}\label{sect-1.4}
It is convenient to work in an abstract algebraic setting. Let $V$ be a real vector space of dimension $n \ge 3,$  with  a
non-degenerate scalar product of signature $(p,q)$:
$$\XIx: V \times V \rightarrow \mathbb{R}\,.$$ Let
$\mathfrak{R}=\mathfrak{R}(V)\subset\otimes^4V^*$ be the space of all {\it generalized curvature tensors}. An element
$A\in\otimes^4(V^*)$ belongs to $\mathfrak{R}$ if and only if $A$ satisfies the relations of  Equation
(\ref{eqn-1.b}) and Equation (\ref{eqn-1.c}). In what follows, we will use $A$ and $\mathcal{A}$, respectively, when working in the
abstract algebraic context, and we will use
$R$ and $\mathcal{R}$, respectively, when working in the geometric context; analoguously
we use $\XIx$ and $g$, respectively, to raise and lower
indices as needed.

The subspace $\mathfrak{W}\subset\mathfrak{R}$ of {\it
Weyl tensors} is  defined by imposing, in addition to the relations of Equations (\ref{eqn-1.b}) 
and (\ref{eqn-1.c}), the symmetry of Equation (\ref{eqn-1.e}). The subspace of
{\it algebraic curvature tensors} $\mathfrak{A}\subset\mathfrak{R}$ is  defined by imposing, additionally the relations of
Equation  (\ref{eqn-1.b}) and of Equation  (\ref{eqn-1.c}), the
symmetry of Equation (\ref{eqn-1.f}); elements of $\mathfrak{A}$ are said to be {\it algebraic}.  Note that
$$\mathfrak{A}\subset\mathfrak{W}\subset\mathfrak{R}\,.$$
We shall see in Section \ref{sect-2} that these are proper containments if $n\ge4$. 

Let $A \in \mathfrak{R}(V)$. In the presence of Equations (\ref{eqn-1.b}) and (\ref{eqn-1.c}) the
relations of Equation (\ref{eqn-1.f}) and the curvature symmetry $A(x,y,z,w) = A(z,w,x,y)$ are equivalent, see, for example,
the discussion in \cite{GNU09}.
Consequently, it is useful to introduce the {\it conjugate curvature tensor}
$A^\star$ by setting:
$$ A^\star(x,y,w,z)=- A(x,y,z,w)\,.$$
Note that the conjugate curvature tensor
does not necessarily satisfy the Bianchi identity given in Equation (\ref{eqn-1.c}) (see, for example, the discussion in
\cite{GNU09} Section 2.3). Consequently,
$A^\star$ in general is not  an element of $ \mathfrak{R}(V)$. 
We raise indices to define $\mathcal{A}^\star$; $\mathcal{A}^\star$ is characterized by the identity:
$$\XIx(\mathcal{A}^\star(x,y)z,w):=-A(x,y,w,z)\,.$$

We introduce a convenient notation from the physics literature and set
$$F:= \, - \, \tfrac{2}{n}\, \Lambda{\operatorname{Ric}}\,.$$ 
Let $\XIx_{ij}:=\XIx( e_i,e_j)$. If
$A\in\mathfrak{W}$, we have:
\begin{eqnarray*}
&&A^\star_{ijkl} + A^\star_{ijlk}= -2F_{ij}\,\XIx_{kl},\\
&&A^\star_{ijkl} + A^\star_{jkil}  +
A^\star_{kijl}= \tfrac{4}{n}\,(\Lambda R_{ij}\,h_{kl} + \Lambda R_{jk}\,h_{il} + \Lambda R_{ki}\,h_{jl}) \\
&&\qquad\qquad\qquad\qquad\phantom{..a}= - \,2 \,(F_{ij}\, \XIx_{kl} + F_{jk} \, \XIx_{il} + F_{ki} \, \XIx_{jl}).
\end{eqnarray*}

The following Proposition was proved in \cite{GNU10}:
\begin{proposition}\label{prop-1.3}
Let $n\ge3$. Let $A\in\mathfrak{W}$. The following assertions are equivalent:
$$\begin{array}{llllll}
(1)&A\in\mathfrak{A}.\quad&(2)&A^\star\in\mathfrak{A}.\quad&(3)&A^\star\text{satisfies 
the Bianchi identity} \enspace (\ref{eqn-1.c}). \end{array}$$
\end{proposition}
The Proposition implies (see \cite{GNU10}):
\begin{theorem}
Let $n\ge3$. Let $\mathcal{W}=(N,g,\nabla)$ be a Weyl manifold. Assume that $H^1(N;\mathbb{R})=0$ and that the
conjugate curvature tensor       $R^\star$ is a generalized curvature tensor. Then there exists $f\in
C^\infty(N)$ so that the Weyl connection 
$\nabla$ is the Levi-Civita connection of the conformally equivalent metric $e^{2f}g$.
\end{theorem}

The definition of the conjugate curvature tensor implies
 $\operatorname{Ric}^\star(R) = \operatorname{Ric}(R^\star)$. 
Moreover, Theorem \ref{thm-1.1} implies:
\begin{proposition} Let $(N,g,\nabla)$ be a Weyl manifold. Then:
\begin{enumerate}
\item If  $n = 4$, then  $\Lambda \operatorname{Ric}^\star = 0$ and thus $\operatorname{Ric}^\star$ is symmetric.
\item  If  $n \ne 4$, then $R^\star$ satisfies the relation:
$$ R^\star_{ijlk} +  R^\star_{ijkl} = \, \tfrac{4}{n-4} \cdot \Lambda R^\star_{ij} \cdot g_{kl}.$$
\end{enumerate}
\end{proposition}

We say that the triple $\mathcal{W}:=(V,h,A)$ is a {\it Weyl model} if $A\in\mathfrak{W}$. We say that such a triple is {\it
geometrically realized} by the Weyl manifold
$\mathcal{W} = (N,g,\nabla)$ if there exists a point $P\in N$ and an isomorphism $\Phi:V\rightarrow T_PN$ so that 
$\Phi^*g_P=\XIx$ and so that
$\Phi^*R_P=A$. One can pass from the algebraic setting to the geometric setting using the following result
\cite{GNU10}.

\begin{theorem}
Every Weyl model is geometrically realized by a Weyl manifold.
\end{theorem}

Here is a brief outline to the remainder of the paper. In Section \ref{sect-2}, we derive the
curvature decomposition of Higa \cite{H94} for Weyl manifolds; we shall discuss Higa's result
in the context of the decomposition results from \cite{BNGS06,GNU09}. In Section \ref{sect-3}, we study further  curvature
properties and various special classes
of  {\it Weyl manifolds}. We shall
apply  Higa's   gauge
invariant {\it canonical  metric} and prove several global results. We shall also
 study  Einstein-Weyl manifolds and
projectively flat Weyl connections. We conclude our discussion in Section \ref{sect-4} applying our decomposition results to the study of 
the well known gauge invariant curvatures, the {\it directional curvature}  and  the {\it length curvature}. Set
$\mathcal{F}(x,y)z:=F(x,y)z$.

\section{\bf Curvature decompositions}\label{sect-2}
We recall some results from  \cite{B90}
related to earlier results of Singer and Thorpe \cite{ST69} (see also the discussion in \cite{GNU09}). Let
$\mathcal{O}:=O(V,\XIx)$ be the associated {\it orthogonal group}.
\begin{definition}
\rm Set\begin{enumerate}
\item $S^2:=\{\theta\in\otimes^2(V^*):\theta_{ij}=\theta_{ji}\}$.
\item
$S^2_0:=\{\theta\in\otimes^2(V^*):\theta_{ij}=\theta_{ji},\quad \XIx^{ij}\theta_{ij}=0\}$.
\item $\Lambda^2:=\{\theta\in\otimes^2(V^*):\theta_{ij}=-\theta_{ji}\}$.
\item $W_6:=\{A\in\otimes^4(V^*):A_{ijkl}=-A_{jikl}=A_{klij},\quad \XIx^{il}A_{ijkl}=0$,
\smallbreak\qquad\qquad\qquad
\hspace{8mm} $A_{ijkl}+A_{jkil}+A_{kijl}=0\}$.
\item $W_7:=\{A\in\otimes^4(V^*):A_{ijkl}=-A_{jikl}=A_{ijlk},\quad \XIx^{il}A_{ijkl}=0,$
\smallbreak\qquad\qquad\qquad
\hspace{8mm}
$A_{kjil}+A_{ikjl}-A_{ljik}-A_{iljk}=0\}$.
\item $W_8:=\{A\in\otimes^4(V^*):A_{ijkl}=-A_{jikl}=-A_{klij},\quad \XIx^{il}A_{ijkl}=0\}$.
\end{enumerate}
Note that $W_6$ and $W_7$ are submodules of $\mathfrak{R}$ whereas $W_8$ is not a submodule of
$\mathfrak{R}$.\end{definition}
\subsection{ Decompositions of $\mathfrak{A}$ and  $\mathfrak{R}$}
\begin{theorem}\label{thm-2.2}
Let $n\ge4$.
\begin{enumerate}
\item The modules $\{\mathbb{R}$, $S_0^2$, $\Lambda^2$, $W_6$, $W_7$, $W_8\}$ are inequivalent and irreducible $\mathcal{O}$
modules.
\item There is an $\mathcal{O}$ module isomorphism 
   $\mathfrak{R}\approx\mathbb{R}\oplus 2\cdot S_0^2\oplus
2\cdot\Lambda^2\oplus W_6\oplus W_7\oplus W_8$.
\item There is an $\mathcal{O}$ module isomorphism
$\mathfrak{A}\approx\mathbb{R}\oplus S_0^2\oplus W_6$.
\end{enumerate}
\end{theorem}

\begin{remark}\rm 
If $n=3$, we set $W_6=W_8=0$ to obtain the corresponding decomposition. For $n\ge5$, the modules of Assertion (1) 
in Theorem \ref{thm-2.2} are also  irreducible $SO(V,\XIx)$ modules; if $n=4$, we must decompose $W_6=W_6^+\oplus
W_6^-$ as the sum of the dual and anti-self dual Weyl conformal curvature tensors. The space $W_6$ is the space of all {\it Weyl
conformal curvature tensors}. One has:
$$\begin{array}{ll}
\dim\{\mathfrak{R}\}=\frac13n^2(n^2-1),&\dim\{\Lambda^2\}=\textstyle\frac{n(n-1)}2,\\
\vspace{1mm}
\dim\{\mathfrak{A}\}=\textstyle\frac1{12}n^2(n^2-1),&\dim\{W_6\}=\textstyle\frac{n(n+1)(n-3)(n+2)}{12},\\
\vspace{1mm}
\dim\{\mathbb{R}\}=1,&\dim\{W_7\}=\textstyle\frac{(n-1)(n-2)(n+1)(n+4)}8,\\
\vspace{1mm}
\dim\{S^2_0)=\textstyle\frac{(n-1)(n+2)}2,&\dim\{W_8\}=\textstyle\frac{n(n-1)(n-3)(n+2)}8\,.
\end{array}$$
\end{remark}

\begin{proof} We sketch the proof of Theorem \ref{thm-2.2}. 
Define $\pi_{\Lambda\otimes S}:\mathfrak{R}(V)\rightarrow\Lambda^2\otimes S^2$ by:
$$\pi_{\Lambda\otimes S}(A)(x,y,z,w):=\textstyle\frac12\{A(x,y,z,w)+A(x,y,w,z)\}\,.$$
One verifies easily that $\ker(\pi_{\Lambda\otimes S})=\mathfrak{A}(V)$. Let $\Theta\in\Lambda^2\otimes S^2$. 
Define
$$\{\sigma_{\Lambda\otimes
S}(\Theta)\}_{ijkl}:=\Theta_{ijkl}+\textstyle\frac12\{\Theta_{kjil}+\Theta_{ikjl}-\Theta_{ljik}-\Theta_{iljk}\}\,.
$$
Let $A=\sigma_{\Lambda\otimes S}\Theta$. Clearly $A_{ijkl}=-A_{jikl}$. We verify the Bianchi
identity is satisfied and show thereby that $\sigma_{\Lambda\otimes S}:\Lambda^2\otimes S^2\rightarrow\mathfrak{R}$ by
computing:
\medbreak\qquad
$A_{ijkl}+A_{jkil}+A_{kijl}
=\Theta_{ijkl}+\textstyle\frac12\{\Theta_{kjil}+\Theta_{ikjl}-\Theta_{ljik}-\Theta_{iljk}\}$
\smallbreak\qquad\quad\ 
$+\Theta_{jkil}+\textstyle\frac12\{\Theta_{ikjl}+\Theta_{jikl}-\Theta_{lkji}-\Theta_{jlki}\}$
\smallbreak\qquad\quad\ 
$+\Theta_{kijl}+\textstyle\frac12\{\Theta_{jikl}+\Theta_{kjil}-\Theta_{likj}-\Theta_{klij}\}$
\smallbreak\qquad\quad
$=\Theta_{ijkl}(1-\frac12-\frac12)+\Theta_{jkil}(1-\frac12-\frac12)+\Theta_{kijl}(1-\frac12-\frac12)$
\smallbreak\qquad\quad\ 
$+\Theta_{ljik}(-\frac12+\frac12)+\Theta_{iljk}(-\frac12+\frac12)+\Theta_{klji}(-\frac12+\frac12)=0$.
\medbreak\noindent
We show that $\sigma_{\Lambda\otimes S}$ is a splitting of $\pi_{\Lambda\otimes S}$ by checking:
\medbreak\qquad
$(\pi_{\Lambda\otimes S}\sigma_{\Lambda\otimes S}\Theta)_{ijkl}$
\smallbreak\qquad\quad
$=\frac12\Theta_{ijkl}+\textstyle\frac14\{\Theta_{kjil}+\Theta_{ikjl}-\Theta_{ljik}-\Theta_{iljk}\}$
\smallbreak\qquad\quad\ 
$+\frac12\Theta_{ijlk}+\textstyle\frac14\{\Theta_{ljik}+\Theta_{iljk}-\Theta_{kjil}-\Theta_{ikjl}\}
=\Theta_{ijkl}$.
\medbreak\noindent 
Consequently, $\pi_{\Lambda\otimes S}$ is a surjective map and thus we have an $\mathcal{O}$ module isomorphism:
$$\mathfrak{R}\approx
\mathfrak{A}\oplus (\Lambda^2\otimes S^2)\,.$$
The Theorem then follows from the decomposition of $\mathfrak{A}$ as an $\mathcal{O}$ module \cite{ST69} and by decomposing
$\Lambda^2\otimes S^2$ as an $\mathcal{O}$ module; we omit details in the interests of brevity.
\end{proof}

\begin{definition}
\rm If $\varphi\in\Lambda^2$, set\begin{enumerate}
\item $\sigma_4(\varphi)(x,y,z,w):=2\varphi(x,y)\XIx( z,w)+\varphi(x,z)\XIx( y,w)-\varphi(y,z)\XIx(x,w)$.
\item $\sigma_5(\varphi)(x,y,z,w):=\varphi(x,w)\XIx( y,z)-\varphi(y,w)\XIx(x,z)$.
\end{enumerate}\end{definition}

\begin{lemma}
If $\varphi\in\Lambda^2$  then $\sigma_4(\varphi)\in\mathfrak{R}$ and $\sigma_5(\varphi)\in\mathfrak{R}$.
\end{lemma}
\begin{proof} Let $\varphi\in\Lambda^2$. Let $A_4:=\sigma_4(\varphi)$ and $A_5:=\sigma_5(\varphi)$. It is then immediate that
$A_i(x,y,z,w)=-A_i(y,x,z,w)$. We establish the Bianchi identity by computing:
\medbreak\qquad
$A_4(x,y,z,w)+A_4(y,z,x,w)+A_4(z,x,y,w)$
\par\qquad\quad$=2\varphi(x,y)\XIx(z,w)
+\varphi(x,z)\XIx( y,w)-\varphi(y,z)\XIx( x,w)$
\par\qquad\quad\phantom{.}$+2\varphi(y,z)\XIx(x,w)
+\varphi(y,x)\XIx( z,w)-\varphi(z,x)\XIx( y,w)$
\par\qquad\quad\phantom{.}$+2\varphi(z,x)\XIx( y,w)
+\varphi(z,y)\XIx( x,w)-\varphi(x,y)\XIx( z,w)=0$,
\smallbreak\qquad$A_5(x,y,z,w)+A_5(y,z,x,w)+A_5(z,x,y,w)$
\par\qquad\quad$=\varphi(x,w)\XIx( y,z)-\varphi(y,w)\XIx( x,z)$
\par\qquad\quad\phantom{.}$+\varphi(y,w)\XIx( z,x)-\varphi(z,w)\XIx( y,x)$
\par\qquad\quad\phantom{.}$+\varphi(z,w)\XIx( x,y)-\varphi(x,w)\XIx( z,y)=0$.
\end{proof}
\subsection{Higa's  decomposition of $\mathfrak{W}$}\label{sect-2.2}
We can now discuss the decomposition of $\mathfrak{W}$ as an $\mathcal{O}$ module. We reformulate a result of Higa
\cite{H94}.
\begin{theorem}\label{thm-2.6}
For $n \ge 4,$ we may decompose
$\mathfrak{W}=\mathfrak{A}\oplus\mathfrak{P}$
where $$\mathfrak{P}=\{(\sigma_4-\sigma_5)\varphi\}_{\varphi\in\Lambda^2}\,.$$
Consequently, $\mathfrak{W}\approx\mathbb{R}\oplus S_0^2\oplus\Lambda^2\oplus W_6$ as an $\mathcal{O}$ module.
\end{theorem}

\begin{proof}  In the proof of Theorem \ref{thm-2.2}, we constructed a short exact sequence
$$0\rightarrow\mathfrak{A}(V)\rightarrow\mathfrak{R}(V)\mapright{\pi_{\Lambda\otimes S}}\Lambda^2\otimes S^2\rightarrow0\,.$$
Equation (\ref{eqn-1.e}) then implies that
$\pi_{\Lambda\otimes S}(A)$ takes values in $\Lambda^2\otimes\mathbb{R}\cdot\XIx$. It is immediate
from the definition that $\mathfrak{A}\subset\mathfrak{W}$, and thus either
$\mathfrak{W}=\mathfrak{A}$ or $\mathfrak{W}$ is isomorphic to $\mathfrak{A}\oplus\Lambda^2$ as an $\mathcal{O}$ module. We argue
that this latter possibility pertains. 
If
$\varphi\in\Lambda^2(V)$, define:
$$
\begin{array}{l}
A_{ijkl}:=(\sigma_4\varphi-\sigma_5\varphi)_{ijkl}\\
\qquad\ \ =2\varphi_{ij} h_{kl}+\varphi_{ik} h_{jl}-\varphi_{jk} h_{il}
   -\varphi_{il} h_{jk}+\varphi_{jl} h_{ik}\,.
\end{array}$$
We define the Ricci type components $A_{jk} := \operatorname{Ric}(A)_{jk} := h^{il}A_{ijkl}$ \, and compute:
\medbreak\qquad
$A_{jk}= \XIx^{il}\{2\varphi_{ij} \XIx_{kl}+\varphi_{ik} \XIx_{jl} - \varphi_{jk} \XIx_{il} - \varphi_{il} \XIx_{jk}
+
\varphi_{jl}\XIx_{ik}\}$
\smallbreak\qquad\qquad
$=2\varphi_{kj} + \varphi_{jk} - n\varphi_{jk} + \varphi_{jk}= - n\varphi_{jk}$,
\smallbreak\qquad
$A_{ijkl} + A_{ijlk} = 4\,\varphi_{ij}\, \XIx_{kl} = \textstyle\frac2n\,(A_{ji} - A_{ij})
\,h_{kl}$.
\end{proof}

Consider Higa's decomposition $\mathfrak{W}=\mathfrak{A}\oplus\mathfrak{P}$ from above and the orthogonal projections $\pi_{\mathfrak{A}}: \mathfrak{W} \rightarrow \mathfrak{A}$ and  $\pi_{\mathfrak{P}}: \mathfrak{W} \rightarrow \mathfrak{P}$.
For $A \in \mathfrak{W}$ we call $H(A): = \pi_{\mathfrak{P}}(A)$
the {\it Higa term  of $A$}. 
 In Section \ref{sect-2.7} we will relate $H(A)$ with the  decompositions that we study in the next Sections \ref{sect-2.3} -
\ref{sect-2.5}.
\subsection{The $A-$decomposition  and the $W-$decomposition
for $\mathfrak{W}$}\label{sect-2.3}
In Theorem \ref{thm-2.2}, we identified the decomposition factors of $\mathfrak{R}$ as an
$\mathcal{O}$ module. However, since the modules $S_0^2$ and $\Lambda^2$ both appear with multiplicity $2$, the decomposition of
$\mathfrak{R}$ is not unique. In
\cite{GNU09} we studied  decompositions of
$\mathfrak{R}$
in some detail and,
following the discussion in \cite{B90}, presented two different possibilities 
for the  decomposition of $\mathfrak{R}$:
$$\bigoplus_{i=1}^8A_i = \mathfrak{R} = \bigoplus_{i=1}^8W_i\,.$$
We denote the corresponding orthogonal projections by
$$\alpha_i: \mathfrak{R} \rightarrow A_i \quad \text{and} \quad
\pi_i: \mathfrak{R} \rightarrow W_i\,.$$

Let $p(A):=\mathfrak{R}\cap\ker(\operatorname{Ric})$ be the space of {\it projective curvature tensors}. While the {\it
$A$-decomposition} also gives rise to a decomposition of $\mathfrak{A}$, it does not induce a decomposition of $p(A)$. On the
other hand the {\it $W$-decomposition} induces a decomposition of $p(A)$ but does not induce a decomposition of $\mathfrak{A}$.
Thus both decompositions are important in the geometric study of manifolds and their curvature properties. Again we emphasize that
this is possible because both $S_0^2$ and $\Lambda^2$ appear with multiplicity $2$ in the decomposition of $\mathfrak{R}$;
thus identifying the exact subspace of $\mathfrak{R}$ which is isomorphic to $S_0^2$ or to $\Lambda^2$ in the submodules
$\mathfrak{A}$,
$p(A)$, and
$\mathfrak{W}$ is crucial.
We follow the discussion in 
\cite{GNU09} of these two inequivalent decompositions and evaluate them for Weyl manifolds.
We begin by establishing some useful notational conventions:
\begin{definition}
\rm
Let $\theta_1$ and $\theta_2$ be bilinear forms.
\begin{enumerate}
\item Define
$(\theta_1\cdot\theta_2)(x,y,z,w):=\theta_1(x,y)\theta_2(z,w)$.
\item Define $(\theta_1 \wedge_r\theta_2) (x,y,z,w):=
\theta_1(x,z)\theta_2(y,w) - \theta_1(y,z)\theta_2(x,w)$
\smallbreak\qquad
$-r[\theta_1(x,w)\theta_2(y,z) - \theta_1(y,w)\theta_2(x,z)]$ for  $r \in \mathbb{N}$,    and $\wedge:=\wedge_0$.
\item Define mappings $\psi$ and $\mu$ from 
$\otimes^4V^*$ to $\otimes^4V^*$  by setting
\begin{eqnarray*}
&&4\psi (A)(x,y,z,w) :=A(x,y,z,w) + A(y,x,w,z)\\
&&\qquad + 
A(z,w,x,y) + A(w,z,y,x),\quad\text{and}\\
&& 8 \mu (A)(x,y,z,w) := 3 A(x,y,z,w) + 3 A(x,y,w,z)\\
&&\qquad+A(x,w,z,y) +A (x,z,w,y) + A (w,y,z,x) + 
A(z,y,w,x)\,.
\end{eqnarray*}
\end{enumerate}
\end{definition}
\begin{remark}\rm\ \begin{enumerate}
\item Note $\theta_1 \wedge_k\theta_1 = (k+1)\theta_1 \wedge\theta_1$.
\item $\theta_1\wedge_r\theta_2\, + \, \theta_1 \wedge_s \theta_2 = \theta_1 \wedge \theta_2 \,+ \, \theta_1 \wedge_{r+s}
\theta_2$.
\end{enumerate}\end{remark}
\subsection{The $A-$decomposition  for $\mathfrak{W}$}
Recall that $\alpha_i$ is orthogonal
projection on the component $A_i$.
\begin{lemma}\label{lem-2.9}
If $A \in \mathfrak{W}$, then:
\begin{enumerate}
\item $\alpha_1({A}) =-\tfrac {1}{n(n-1)}\tau_h \,  \XIx \wedge  \XIx$.
\medbreak\item
$\alpha_2({A}) = \tfrac{-1}{2(n-2)} \,S(\operatorname{Ric} +\operatorname{Ric}^\star)\wedge_1 \XIx + \tfrac{2
}{n(n-2)}\tau_h
\,  \XIx \wedgeo  \XIx$
\smallbreak
$\phantom{...A}= - \,\tfrac{1}{n-2}
\left(S\operatorname{Ric} \wedge_1 \XIx\right)  + \tfrac{2 }{n(n-2)}\tau_h \XIx \wedge \XIx$.
\medbreak\item
$\alpha_3({A}) = - \, \tfrac{1}{2n}\, S(\operatorname{Ric} - \operatorname{Ric}^\star)\wedge_{-1} \XIx = 0$.
\medbreak\item
$\alpha_4({A}) = - \, \tfrac{1}{4(n+2)}\,  [2\Lambda (3 \operatorname{Ric} -  \operatorname{Ric}^\star)
\cdot \XIx + 
\Lambda (3 \operatorname{Ric} -  \operatorname{Ric}^\star) \wedge_{-1} \XIx] $
\smallbreak
$\phantom{...A} = - \, \tfrac{1}{2n}\, \left(2(\Lambda \operatorname{Ric})\cdot \XIx + (\Lambda \operatorname{Ric})\wedge_{-1}\XIx
\right)$.
\medbreak\item
$\alpha_5({A}) = - \, \tfrac{1}{4(n-2)} \, [2\Lambda (\operatorname{Ric} + \operatorname{Ric}^\star) \cdot \XIx + 
\Lambda ( \operatorname{Ric}  + \operatorname{Ric}^\star) \wedge_3 \XIx] $
\smallbreak
$\phantom{...A}= - \, \tfrac{1}{2n} \, \left(2(\Lambda \operatorname{Ric})\cdot \XIx + (\Lambda \operatorname{Ric})\wedge_{3}\XIx
\right)$.
\medbreak\item
$\alpha_6({A}) = \psi({A}) - \alpha_1({A}) 
- \alpha_2({A}) $
\smallbreak
$ = A + \tfrac{2}{n}\, (\Lambda \operatorname{Ric})\cdot \XIx + 
\tfrac{1}{n}\, (\Lambda \operatorname{Ric})\wedge_1\XIx + \tfrac{1}{n-2}\,(S\operatorname{Ric})\wedge_1\XIx -
\tfrac{1}{(n-1)(n-2)}\tau_h\XIx
\wedge \XIx$. 
\medbreak\item
$\alpha_7({A})  = 0$.
\medbreak\item
$\alpha_8({A})  = 0$.
\end{enumerate}
\end{lemma}
\begin{remark}\rm
Let $A \in \mathfrak{W}$. In analogy to the Schouten tensor in conformal
semi-Riemannian geometry introduce the symmetric 
{\it Weyl-Schouten tensor} by setting:
$$ \sigma := \tfrac{1}{n-2}\, \left [ S\, \operatorname{Ric} - \tfrac{1}{2(n-1)} \cdot \tau_h  \XIx \right ].
$$
We then have $ \alpha_6(A) = A + \sigma \wedge_1 \XIx - (\alpha_4(A) + \alpha_5(A))$.
This decomposition extends the well known  decomposition of the 
Weyl conformal curvature tensor in 
semi-Riemannian geometry; note that $\alpha_6(A)$ and also
$\sigma \wedge_1 \XIx$ are algebraic curvature tensors.
We may express
$\alpha_2({A}) =  \,- \,  \sigma \wedge_1 \, \XIx\,+\,\tfrac{1}{n(n-1)}\tau_h \XIx \wedge \XIx$.
\end{remark}
Lemma \ref{lem-2.9} implies:
\begin{lemma}
If $A \in \mathfrak{W}$ then:
\begin{enumerate}
\item $\operatorname{Ric}(\alpha_1({A})) = \operatorname{Ric}^\star(\alpha_1(A))=\tfrac{1}{n}\tau_h\XIx$.
\vspace{1mm}
\item $\operatorname{Ric}(\alpha_2({A})) = \operatorname{Ric}^\star(\alpha_2(A))=-\tfrac{1}{n}\tau_h\XIx +  S(\operatorname{Ric})$.
\vspace{1mm}
\item $\operatorname{Ric}(\alpha_4({A})) = -\operatorname{Ric}^\star(\alpha_4(A))= \tfrac{n+2}{2n} \Lambda(\operatorname{Ric})$.
\vspace{1mm}
\item $\operatorname{Ric}(\alpha_5({A})) = \frac13\operatorname{Ric}^\star(\alpha_5(A))=
\tfrac{n-2}{2n}\Lambda(\operatorname{Ric})$.
\vspace{1mm}
\item $ \operatorname{Ric}(\alpha_j({A})) = \operatorname{Ric}^\star(\alpha_j({A})) = 0$  for  $j = 3,6,7,8$.
\end{enumerate}
\end{lemma}
\subsection{The $W-$decomposition  for $\mathfrak{W}$}\label{sect-2.5}
Recall that $\pi_i$ is orthogonal projection on the component $W_i$.
\begin{lemma}
If $A \in \mathfrak{W}$ then:
\begin{enumerate}
\item $\pi_1(A) =\tfrac {1}{n(n-1)}\tau_h \XIx \wedgeo  \XIx$.
\smallbreak\item $\pi_2(A) = \tfrac {1}{n-1}\, [\frac{1}{n}\tau_h \XIx -S\operatorname{Ric}] \wedgeo  \XIx$.
\smallbreak\item 
$\pi_3(A) = - \, \tfrac {1}{n+1}\,  [ 2 \Lambda \operatorname{Ric}  \cdot \XIx + \Lambda \operatorname{Ric}
 \wedgeo  \,\XIx]$.
\smallbreak\item $\pi_4(A) = - \,
 \tfrac {1}{n^2-4}\,[2 \Lambda \operatorname{Ric}^{*} \cdot \XIx 
 + 
\Lambda \operatorname{Ric}^{*} \wedge_{n+1} \XIx ]$
\medbreak
$\, -\tfrac {3}{(n^2 - 4)(n+1)}\, [2 \Lambda \operatorname{Ric}  \cdot \XIx 
+ \Lambda \operatorname{Ric} \wedge_{n+1} \XIx ]$
\medbreak
$ = \, -\tfrac {1}{n(n+1)}\,\left[2(\Lambda \operatorname{Ric})\cdot \XIx +  \Lambda \operatorname{Ric}
\wedge_{n+1}\XIx \right]$.
\medbreak\item
$\pi_5(A) =  \tfrac {1}{(n-1)(n-2)}\,[\tau_h \XIx \wedge  \XIx
-\tfrac{1}{n}\, S(\operatorname{Ric} + (n-1)\operatorname{Ric}^\star) \wedge_{n-1} \XIx]$
\medbreak
$= \, \tfrac {1}{(n-1)(n-2)}\;[\tau_h \XIx \wedge  \XIx - S \operatorname{Ric} \wedge_{n-1} \XIx]$.
\medbreak\item
$\pi_6(A) 
=\psi(A) + \tfrac{1}{2(n-2)} \, S(\operatorname{Ric} + \operatorname{Ric}^\star)\wedge_1 \XIx
- \tfrac{1}{(n-1)(n-2)}\tau_h\XIx \wedgeo  \XIx$
\medbreak
$ = \, A + \tfrac{2}{n}\, (\Lambda \operatorname{Ric})\cdot \XIx + \tfrac{1}{n}\,
(\Lambda \operatorname{Ric}) \wedge_1 \XIx + \tfrac{1}{n-2}\, S \operatorname{Ric} \wedge_1 \XIx 
- \tfrac {1}{(n-1)(n-2)}\tau_h \XIx  \wedge \XIx$.
\medbreak\item
$\pi_7(A)= \mu(R) +\tfrac{1}{2n}\,S(\operatorname{Ric} -  \operatorname{Ric}^\star)\wedge_{-1} \XIx +\tfrac{1}{2(n+2)}\,\Lambda(3
\operatorname{Ric} - \operatorname{Ric}^\star)\cdot \XIx $
\medbreak
$+
\tfrac{1}{4(n+2)}\,\Lambda(3 \operatorname{Ric} - \operatorname{Ric}^\star)\wedge_{-1} \XIx = 0$.
\medbreak\item
$ \pi_8(A)= A - \psi(A) - \mu(A) +
\tfrac{1}{2(n-2)}\,\Lambda(\operatorname{Ric} +  \operatorname{Ric}^\star)\cdot \XIx $
\medbreak\qquad\qquad\qquad $+\tfrac{1}{4(n-2)}\,\Lambda(\operatorname{Ric} +  \operatorname{Ric}^\star)\wedge_{3} \XIx= 0$.
\end{enumerate}\end{lemma}
\begin{remark}\rm
 If $A \in \mathfrak{W}$ satisfies $\operatorname{Ric}(A) = 0$ 
 then
$\pi_i(A) = 0$ for $i \ne 6$ and $\pi_6(A) = A$.
\end{remark}
\begin{lemma}
If $A \in \mathfrak{W}$  then:
\begin{enumerate}
\smallbreak\item $\operatorname{Ric}(\pi_1({A})) = \operatorname{Ric}^\star(\pi_1(A))=\tfrac{1}{n}\tau_h\XIx$.
\smallbreak\item $\operatorname{Ric}(\pi_2({A})) = S \operatorname{Ric} - \tfrac{1}{n}\tau_h \XIx= -(n-1)
\operatorname{Ric}^\star(\pi_2(A))$.
\smallbreak\item $\operatorname{Ric}(\pi_3({A})) = \Lambda \operatorname{Ric} 
= - \,\tfrac{n+1}{3} \operatorname{Ric}^\star(\pi_3(A))$.
\smallbreak\item $\operatorname{Ric}(\pi_4({A})) = 0$.
\smallbreak\item $\operatorname{Ric}^\star(\pi_4(A)) = \tfrac{(n-2)(n+2)}{n(n+1)}\Lambda \operatorname{Ric}$. 
\smallbreak\item $\operatorname{Ric}(\pi_5({A})) = 0$.
\smallbreak\item $\operatorname{Ric}^\star(\pi_5({A})) = \tfrac{1}{n-1} (n S \operatorname{Ric} - \tau_h\XIx) =
\tfrac{n}{n-1}\operatorname{Ric}(\pi_2({A})) = - n \operatorname{Ric}^\star(\pi_2(A))$.
\smallbreak\item $\operatorname{Ric}(\pi_6({A})) = 0$. 
\smallbreak\item $\operatorname{Ric}^\star(\pi_6({A})) = 0$.
\end{enumerate}
\end{lemma}
\begin{remark}\rm
 The Ricci and Ricci${}^\star$ tensors of $\pi_i(A)$ and $\alpha_i(A)$
for $i = 1,\dots,6$ have essentially only 3 non-trivial (i.e. non-vanishing) types:
\begin{enumerate}
\smallbreak\item Constant multiples of $\Lambda(\operatorname{Ric}(A))$.
\smallbreak\item Constant multiples of 
$\tau_h\XIx$.
\smallbreak\item Constant multiples of $(S\operatorname{Ric}(A) -  \tfrac{1}{n}\tau_h\XIx)$.
\end{enumerate}\end{remark}
\subsection{The projective curvature tensor}\label{sect-2.6}
From \cite{B90} and \cite{GNU09}  one knows that the projective curvature tensor $p(A)$ of $A \in \mathfrak{R}$ can be recovered  from the $W$-decomposition as follows:
\begin{equation}\label{eqn-2.a}
p(A) = \bigoplus_{i=4}^8 \pi_i(A)\,.
\end{equation}
Let $A \in \mathfrak{W}$. 
Equation (\ref{eqn-2.a}) and the results for $\pi_i$ given above then yield easily the following result:
\begin{lemma} Let $A \in \mathfrak{W}$; then:
\begin{enumerate}
\item $p(A) = \bigoplus_{i=4}^6 \pi_i(A)$.
\smallbreak\item
$p(A) = A + \tfrac{1}{n+1} \,\left[2  \Lambda \operatorname{Ric}\cdot \XIx + \Lambda \operatorname{Ric} \wedge \XIx \right] + 
\tfrac{1}{n-1} \,(S \operatorname{Ric}) \wedge \XIx\,.$
\end{enumerate}
\end{lemma}
\begin{proposition}\label{prop-2.17}
Let $A \in \mathfrak{W}$. If $p(A) = 0$, then
$A = -\tfrac {1}{n(n-1)}\tau_h\XIx \wedge  \XIx$.
\end{proposition}
\begin{proof} $p(A) = 0$ implies $\pi_i(A) = 0$ for $i = 4,5,6$;
but $\pi_4(A) = 0$ gives $\Lambda \operatorname{Ric} = 0$, and 
$\pi_5(A) = 0$ gives $S\operatorname{Ric} = \tfrac{1}{n}\tau_h\XIx \,;$
this proves the Proposition.
\end{proof}
\subsection{The Higa term and the conjugate curvature tensor}\label{sect-2.7}
We now relate the conjugate curvature tensor $R^\star$ and  the {\it Higa term}  $H(A)$ of $A$. As the algebraic part of $A \in \mathfrak{W}$
is given by $(\alpha_1 + \alpha_2 + \alpha_6)(A)$, 
Lemma \ref{lem-2.9} gives:
\begin{lemma}
\ \begin{enumerate}
\item $H=H(A)=(\alpha_4 + \alpha_5)(A)  \in \mathfrak{P}$. 
\item $A = (\alpha_1 + \alpha_2 + \alpha_6)(A)  + H(A)$.
\end{enumerate}
\end{lemma}
Moreover, we set
$$D(x,y,z,w) := - \,(A(x,y,z,w) + A(x,y,w,z))\,.$$ 
Then the conjugate curvature tensor satisfies 
$$A^\star(x,y,z,w) = A(x,y,z,w) + D(x,y,z,w)\,.$$
We use Equation (\ref{eqn-1.e}) to express:
$$ - D(x,y,z,w) = A(x,y,z,w)+A(x,y,w,z)=  - \tfrac4n\, (\Lambda\operatorname{Ric}(A))(x,y)\,\XIx(z,w)\,$$
and  note the symmetries:
$$D(x,y,z,w) = D(x,y,w,z) = - D(y,x,z,w)\,.$$
\begin{lemma} Adopt the notation established above. Then:
\begin{enumerate}
\item
$ - \, 4 H(x,y,z,w) = 
 2 D(x,y,z,w) + D(x,z,y,w)$
\smallbreak\qquad\qquad$ - D(y,z,x,w) - D(x,w,y,z)
+ D(y,w,x,z)$
\smallbreak\qquad\quad
$=2 D(x,y,z,w) - \left[D(x,y,z,w) + D(y,z,x,w) + D(z,x,y,w) \right]$
\smallbreak\qquad\qquad
$- \left\{D(x,w,y,z) +  D(w,y,x,z) + D(y,x,w,z) \right\}$.
\item $D(x,y,z,w) = H(x,y,z,w) +  H(x,y,w,z).  $
\end{enumerate}
\end{lemma}
In particular, (1) in the Lemma implies that  $D$  determines
the Higa term $H$, while (2) implies that $H$ determines $D$.
In (1), note that  the brackets $[\dots]$ are cyclic in $(x,y,z)$ and $\{\dots\}$ 
are cyclic in $(x,w,y)$. This immediately gives the proof of (2).
Finally, the definition of $A^\star$, the Higa decomposition of $A$,
and the symmetries of the algebraic part of $A$ yield: 
\begin{lemma}
\ \begin{enumerate}
\item
 $A^\star(x,y,z,w)  +A^\star(x,y,w,z) = D(x,y,w,z).$
\item
$ A^\star(x,y,z,w) = (\alpha_1 + \alpha_2 + \alpha_6)(A)(x,y,z,w) 
- (\alpha_4 + \alpha_5)(A)(x,y,w,z) $ \\
$ = (\alpha_1 + \alpha_2 + \alpha_6)(A)(x,y,z,w) - H(A)(x,y,w,z).$
\end{enumerate}
\end{lemma}

\begin{proposition} 
Let $A \in \mathfrak{W}$, then the Higa term satisfies:
\begin{eqnarray*}
H(A)&=&  \alpha_4(A) +  \alpha_5(A)
 = \pi_3(A) + \pi_4(A)\\
& =&
-\, \tfrac{1}{n} \, \left( 2(\Lambda \operatorname{Ric}) \cdot \XIx 
+ (\Lambda \operatorname{Ric}) \wedge_1 \XIx  \right)\,.
\end{eqnarray*}
\end{proposition}
\begin{lemma}\label{lem-2.22} Let $A \in \mathfrak{W}$. For $n \ne 2$ we have the following equivalences:
$$\begin{array}{llllllll} 
(1)&\Lambda\operatorname{Ric} = 0. & 
(2)&\alpha_4(A) = 0.&
(3)&\alpha_5(A)= 0.\phantom{.....}
(4)&H(A) = 0.\\
(5)&\pi_3(A) = 0.&
(6)&\pi_4(A) = 0.&
(7)&A\in\mathfrak{A}.\end{array}$$
\end{lemma}

\section{\bf   Weyl Manifolds}\label{sect-3}
\subsection{Equivalent notions for Weyl manifolds}
If $\phi$ is a smooth $1$-form on a semi-Riemannian manifold $(N,g)$, the dual vector field $\phi^{\sharp}$ is characterized
by the identity 
$g(x,\phi^{\sharp}) = \phi(x)$  for
all tangent fields $x$. The following result (see, for example, Theorem 6 \cite{GNU10}) can be used to construct Weyl manifolds:
\begin{theorem}\label{thm-3.1}
Let $\nabla$ be a torsion free connection on a semi-Riemannian manifold $(N,g)$. Let $\phi$ be a smooth $1$-form on $N$. The
following assertions are equivalent, and if either is satisfied  then  $(N,g,\nabla)$ is a Weyl manifold:
\begin{enumerate}
\item $\nabla g=-2\phi\otimes g$.
\item
$\nabla_xy=\nabla_x^gy+\phi(x)y+\phi(y)x-g(x,y)\phi^{\sharp}$.
\end{enumerate}
\end{theorem}
\subsection{Weyl manifolds and curvature decompositions}\label{sect-3.2}
We pass from the algebraic setting to the geometric setting;
for each $P \in N$ the semi-Riemannian metric induces a
scalar product on the tangent space, we simply identify
 $(T_PN,g) =:(V,h)$.
A conformal change of the metric does not change the associated orthogonal group; $O(V,g)=O(V,c\cdot g)$ for any real $c\ne 0$. Moreover,
for any metric $g$ within the conformal class and at any $P \in N$
the 
decompositions of the  Weyl curvature tensor $R$ are bijectively associated to corresponding  decompositions of the Weyl curvature operator $\mathcal{R}$, and these decompositions obviously are gauge invariant.
We speak about the $A$-decomposition and the $W$-decomposition
of $\mathcal{R}$, respectively.
Furthermore, since the Ricci tensor is a $\operatorname{GL}$ invariant, it is a gauge invariant.
These observations and our 
 previous discussions immediately give:
\begin{theorem}
Let $\mathcal{N}=(N,g,\nabla)$ be a Weyl manifold and let $\mathcal{R}=\mathcal{R}(\nabla)$. Then:\begin{enumerate}
\item The $A$-decomposition of $\mathcal{R}$ is gauge invariant. $\mathcal{N}$ is Ricci flat if and only if $R=\alpha_6(R)$ for some (and
hence any)  $g$ within the conformal class.
\item The $W$-decomposition of $\mathcal{R}$ is gauge invariant. $\mathcal{N}$ is Ricci flat if and only if $R=\pi_6(R)$ for some
(and hence any) $g$ within the conformal class.
\end{enumerate}\end{theorem}
\subsection{The second Bianchi identity}
We recall the following well known result from \cite{E}, p. 56:
{\it The curvature tensor of any torsion free connection satisfies the second Bianchi identity.} However, since $\nabla_k
\,g^{rs} = 2\phi_k \,g^{rs}$, raising and lowering indices need not commute with $\nabla$-covariant differentiation. One
therefore has:
\begin{lemma}\label{lem-3.3}
Let $\mathcal{N}=(N,g,\nabla)$ be a Weyl manifold and let $R=R(\nabla)$. Then:\begin{enumerate}
\item 
$\nabla_m\, \mathcal{R}_{ijk}{}^{l} + \nabla_i\,\mathcal{R}_{jmk}{}^{l} + \nabla_j\,\mathcal{R}_{mik}{}^{l} = 0$.
\item $
\nabla_m\,R_{ijkl} + \nabla_i\,R_{jmkl} + \nabla_j\,R_{mikl} =
-2( \phi_m\,R_{ijkl}  + \phi_i\,R_{jmkl}    + \phi_j\,R_{mikl} )$.
\end{enumerate}\end{lemma}
The following result is a consequence of Theorem \ref{thm-1.1} and of Lemma \ref{lem-3.3}:
\begin{corollary}
We have the  relations: 
\begin{eqnarray*}
&& \nabla_mR_{jk} \; - \; \nabla_jR_{mk} \; + \; \nabla_iR_{jmk}{}^i \;\; = 0,\\
&&\nabla_m(\Lambda R_{jk}) + \nabla_j(\Lambda R_{km})  + \nabla_k(\Lambda R_{mj}) = 0.
\end{eqnarray*}
\end{corollary}
\subsection{The canonical  Weyl metric}
Let $\mathcal{W} = (N,g,\phi)$ be a Weyl manifold. If $g_1: = e^{2f}g$ is a conformally
equivalent metric, $\tau_{g_1}:= e^{-2f}\tau_g$. Thus  there is a gauge invariant, disjoint
decomposition of $N$ into three 
 subsets
$N = N_0 \cup N^+ \cup N^-$ where  
$$N_0 := \{p\in N \,:\, \tau_g(p) = 0\}\quad\text{and}\quad 
N^\pm := \{p\in N\, :\,\pm\tau_g(p) > 0\}\,.
$$ In the following we consider Weyl manifolds
with
$ \tau_g \ne 0$ and restrict to the case $\tau_g >  0$, thus $N = N^+$; the case $\tau_g<0$ can be handled
analogously. We recall a definition of Higa
\cite{H94}; his definition and our definition differ by a constant
positive factor; Higa calls his  metric  the {\it canonical metric}; we use the following terminology:
\begin{definition}
\rm\ \begin{enumerate}
\item Let $\mathcal{W} = (N,g,\nabla)$ be a Weyl manifold with $\tau_g >  0$.
We call the {\it gauge invariant metric} 
$\tilde g : = \tau_g\cdot g$
the {\it canonical Weyl metric} of $\mathcal{W}$. The metric $\tilde g$ is the unique element in the conformal class
defined by $g$ so that $\tau_{\tilde g}=1$.
\item  Let $(g,\phi)$ be a pair that generates $\mathcal{W}$ with $\tau_g > 0$. We call
$\tilde\phi := \phi - \tfrac{1}{2} d \ln \tau_g$
the gauge invariant Weyl 1-form of  $\mathcal{W}$; $\nabla\tilde g=-2\tilde\phi\otimes\tilde g$.
\end{enumerate}
\end{definition}
\begin{remark}
\rm The definition of the canonical Weyl metric is 
gauge invariant (thus $\tilde g$ is a distinguished metric
 in the conformal class) and therefore  all invariants of its
induced semi-Riemannian geometry are gauge invariant; the analogue is true
for invariants constructed from the pair $(\tilde g, \tilde\phi)$.
\end{remark}
\subsection{Curvature invariants of the Weyl  metric}\label{sect-3.5}
Throughout this section, we assume $\tau_g>0$ and let $\tilde g:=\tau_g\cdot g$ be the Weyl metric. We restrict to
the Riemannian setting i.e. $g$ is positive definite. We introduce the following notational conventions:
\begin{definition}
Let $g$ be a Riemannian metric. Let $\Xi_g:=\det(g_{ij})^{1/2}$.\begin{enumerate}
\item Let $\Delta_g:=\Xi_g^{-1}\partial_{x_i}\Xi_g g^{ij}\partial_{x_j}$ be the Laplace-Beltrami operator.
\item Let $\kappa_g:=\textstyle\frac1{n(n-1)}\operatorname{Tr}_g\operatorname{Ric}_g$ be the normalized scalar curvature of $g$.
\item Let $\omega_g=\Xi_g dx^1\dots dx^n$ be the Riemannian measure.
\end{enumerate}
\end{definition}
The following result gives the {\it conformal scalar curvature relations}:
\begin{proposition}\label{prop-3.8}
\ \begin{enumerate}
\item $n\kappa_{\tilde g} = n\tau_g^{-1}\kappa_g - \tau_g^{-2}\Delta_g\,\tau_g 
-\tfrac14 \tau_g^{-3}\,(n-6)\,\|\grad_g\tau_g\|_g^2$.
\item $1 = \tau_{\tilde g}=
n(n-1)\kappa_{\tilde g} - 2(n-1)\nabla^{\tilde g}_k\tilde\phi^k -(n-1)(n-2) \|\tilde\phi\|_{\tilde g}^2$.
\end{enumerate}\end{proposition}

\begin{proof} These formulas can, of course, be derived from classical formulas in the literature. It is instructive, however, to
give a direct derivation. We argue as follows to prove Assertion (1). Let $g$ be a Riemannian metric. Let $\vartheta$ be a positive
smooth function on
$M$ and let
$g_1:=\vartheta g$. We must express $\kappa_{g_1}$ in terms of $\kappa_g$. Choose local coordinates so the Christoffel symbols of
$\nabla^g$ vanish at the point
$P$ in question. A bit of thought then expresses $R^{g_1}=R^g+\mathcal{E}_1+\mathcal{E}_2$ at $P$ where $\mathcal{E}_1$ is
quadratic in the first derivatives of $\vartheta$ and $\mathcal{E}_2$ is linear in the second derivatives of $\vartheta$ at $P$,
respectively. This leads to a formula of the form:
\begin{equation}\label{eqn-3.a}
\kappa_{g_1}=\vartheta^{-1}\kappa_g+a_1\vartheta^{a_2}\Delta_g\vartheta+a_3\vartheta^{a_4}\|\grad_g \vartheta\|_g^2
\end{equation}
where $a_i=a_i(n)$ are certain universal constants which depend on the dimension, but not on the metric chosen, and which need to
be determined; Assertion (1) will then follow by specializing to the case $\vartheta=\tau_g$. To determine these constants
in the general setting, we may take
$g=dx_1^2+...+dx_n^2$ to be flat. We take as a special case
$\vartheta=\vartheta(x_1)$ and determine the coefficients in Equation (\ref{eqn-3.a}) by computing:
\medbreak\qquad
$\Gamma_{111}=\Gamma_{1ii}=\Gamma_{i1i}=-\Gamma_{ii1}=\frac12\vartheta^\prime$ for $1<i$,
\smallbreak\qquad
$\mathcal{R}^{g_1}(\partial_1,\partial_i)\partial_i=\{-\frac12\vartheta^{-1}\vartheta^{\prime\prime}+\frac12\vartheta^{-2}\vartheta^\prime
\vartheta^\prime\}\partial_1$ for
$1<i$,
\smallbreak\qquad
$\mathcal{R}^{g_1}(\partial_i,\partial_j)\partial_j=-\frac14\vartheta^{-2}\vartheta^\prime \vartheta^\prime\partial_i$ for $1<i<j$,
\smallbreak\qquad
$\operatorname{Tr}_{g_1}\operatorname{Ric}_{g_1}
=-(n-1)\vartheta^{-2}\vartheta^{\prime\prime}+\vartheta^{-3}\{(n-1)-\frac14(n-1)(n-2)\}\vartheta^\prime
\vartheta^\prime$
\smallbreak\qquad\qquad
$=(n-1)\{-\vartheta^{-2}\Delta_g\vartheta-\frac{n-6}4\vartheta^{-3}\|\grad_g\vartheta\|_g^2\}$,
\smallbreak\qquad
$n\kappa_{g_1}=-\vartheta^{-2}\Delta_g\vartheta-\frac{n-6}4\vartheta^{-3}\|\grad_g\vartheta\|_g^2\,.$
\medbreak\noindent This completes the proof of Assertion (1) by showing:
$$a_1=-1,\quad a_2=-2,\quad a_3=-\textstyle\frac{n-6}4,\quad a_4=-3\,.$$

\medbreak Let $(M,g,\nabla)$ be an arbitrary Weyl manifold. We use a similar argument to prove Assertion (2). Again, a bit of
thought shows there are universal constants so
\begin{equation}\label{eqn-3.b}
\tau_g=n(n-1)\kappa_g+a_5\nabla_k\phi^k+a_6\|\phi\|_g^2
\end{equation}
where $a_i=a_i(n)$; Assertion (2) will follow by taking the special case where the metric is $\tilde g$. Again, we evaluate these
constants using the method of universal examples. We take the reference background metric to be flat and $\phi=\vartheta(x_1)dx^1$
for some smooth function $\vartheta$ of one variable. We evaluate the universal coefficients in Equation (\ref{eqn-3.b})
and complete the proof of Assertion (2) by calculating:
\medbreak\qquad
$\nabla_xy=\nabla^g_xy + \phi(x)y + \phi(y)x - 
 g(x,y)\phi^{\sharp}$,
\smallbreak\qquad
$\Gamma_{11}{}^1=\Gamma_{1i}{}^i=\Gamma_{i1}{}^i=-\Gamma_{ii}{}^1=\vartheta$ for $1<i$,
\smallbreak\qquad
$\mathcal{R}(\partial_1,\partial_i)\partial_i=-\vartheta^\prime\partial_1$ for $1<i$,
\smallbreak\qquad
$\mathcal{R}(\partial_i,\partial_j)\partial_j=-\vartheta^2\partial_i$ for $1<i<j$,
\smallbreak\qquad
$\tau_g=-2(n-1)\vartheta^\prime-(n-1)(n-2)\vartheta^2$
\smallbreak\qquad\phantom{...}
$=-2(n-1)\nabla_k^g\phi^k-(n-1)(n-2) \|\phi\|_g^2$.
\medbreak\noindent This completes the proof by showing $a_5=-2(n-1)$ and $a_6=-(n-1)(n-2)$.
\end{proof}

We defined the Weyl metric $\tilde g$ by requiring that $\tilde g$ is in the conformal class of $g$ and so that $\tau_{\tilde g}=1$. 
Conversely, of course, if $g_1$ is in the conformal class of $g$ and if $\tau_{\tilde g}=c$ is constant, then $g_1$
is homothetic to
$\tilde g$. We note that the associated Weyl-Schouten tensor is a gauge invariant where
$$\tilde{\sigma}= \, \tfrac{1}{n-2} \left[ S\operatorname{Ric} - \tfrac{1}{2(n-1)} \;  \tilde g \right]\,.$$
\subsection{Global characterizations of the Weyl metric} Again, throughout this section, we assume $\tau_g > 0$ and
normalize the metric so $\tau_{\tilde g}=1$.
We apply the relations given in Proposition \ref{prop-3.8} to obtain global characterizations
of the Weyl metric within the conformal class in terms of the scalar curvatures
$\kappa_{\tilde g}, \kappa_g$ and $\tau_g$. We assume $g$ is positive definite. We first examine the compact case:

\begin{theorem} Let $N$ be a compact Riemannian manifold. 
The gauge invariant total scalar curvature of $\tilde g$ gives an
upper bound for the right hand side in terms of an arbitrary metric $g$ within the conformal class, i.e.
$$\int_N \kappa_{\tilde g} \; \omega_{\tilde g}  \ge \int_N \kappa_g \cdot \tau_g^{\tfrac{n-2}{2}}\;\omega_g\,.$$
Equality holds if and only if $\tilde g = \tau_g \cdot g$ with $\tau_g = const.$, i.e. both metrics are
homothetic.
\end{theorem}
\begin{proof} The Riemannian volume forms are related by the identity: 
$\omega_{\tilde g}  = \tau_g^{\tfrac{n}{2}} \cdot \omega_g$.
This leads to the relation:
\begin{equation}\label{eqn-3.c}
\begin{array}{l}\displaystyle n \int_N \kappa_{\tilde g} \; \omega_{\tilde g}  = n \int_N \kappa_g 
\;\tau_g^{\tfrac{n-2}{2}}
\;
\omega_g -
\int_N \tau_g^{\tfrac{n-4}{2}}
\Delta_g
\tau_g  \;\omega_g\\
\qquad\qquad\qquad\displaystyle - \tfrac{1}{4} (n-6) \int_N \tau_g^{\tfrac{n-6}{2}} \cdot \|\grad_g\tau_g
\|_g^2 
\;
\omega_g\,.\end{array}
\end{equation}
Let $\alpha:=\frac12(n-2)$. We then have:
\begin{equation}\label{eqn-3.d}
\Delta_g \tau_g ^{\alpha} = \alpha(\alpha -1)
 \tau_g^{\alpha - 2} \; \|\grad_g\tau_g \|_g^2 + \alpha
\tau_g^{\alpha - 1} 
\;\Delta_g
\tau_g
\,.
\end{equation}
We use Equation (\ref{eqn-3.d}) to see that:
\begin{equation}\label{eqn-3.e}
0 = \int_N \tau_g^{\tfrac{1}{2}(n-4)} \; \Delta_g \tau_g \; \omega_g + 
\tfrac{1}{2}(n-4) \int_N \tau_g^{\tfrac{n-6}{2}}\; \|\grad_g\tau_g   \|_g^2 \;\omega_g.
\end{equation}
Similarly  Equation (\ref{eqn-3.c}) and Equation (\ref{eqn-3.e}) give:
$$n\, \int_N \kappa_{\tilde g} \;\omega_{\tilde g}  = n\, \int_N \kappa_g\; 
 \tau_g^{\tfrac{n-2}{2}}\;\omega_g + \tfrac{1}{4}\,(n-2)\,\int_N \tau_g^{\tfrac{1}{2}(n-6)}\; \| \grad_g
\tau_g   \|_g^2
\;\omega_g. 
$$
As $\tau_g > 0  $ by assumption, this implies the inequality; the discussion of equality then follows immediately. 
\end{proof}

Next, we study complete manifolds. In dimension $n \le 6$, the foregoing relations allow to apply
the well known maximum principle of Omori \cite{O67} and Yau \cite{Y75} and also
Yau's \cite{Y76} extension of the harmonic map principle to complete manifolds.
\begin{theorem}
Let $3 \le n \le 6$. Assume that the Riemannian manifold $(N,g)$ is complete with non-negative Ricci tensor
$\operatorname{Ric}_g$. Assume that $\tau_g \in L^p$ for some $p>1$.
\begin{enumerate}
\item If the scalar curvatures satisfy the inequality
$\kappa_{\tilde g} - \tau_g^{-1}\kappa_g \le 0$  then $\tilde g = \tau_g \cdot g$ 
with $\tau_g = const$,
i.e. both metrics are homothetic.
\item 
If there exists $0 \le c \in \mathbb{R}$ with $\kappa_{\tilde g} - \tau_g^{-1}\kappa_g= c$, then
$\tilde g = \tau_g \cdot g$ with $\tau_g = const$, i.e. both metrics are homothetic.
\end{enumerate}\end{theorem}
\begin{proof} Suppose that $\kappa_{\tilde g}-\tau_g^{-1}\kappa_g\le0$. Since $(n-6)\le0$, it then follows from Proposition
\ref{prop-3.8} that $\Delta_g\tau_g\ge0$. Since $\tau_g$ is in $L^p$ for some $p>1$, the results of Yau cited above then show that
$\tau_g$ is constant thus verifying Assertion (1).

 To prove
Assertion (2), we apply the Omori-Yau maximum principle to see that there exists a sequence
 of points of the manifold such that $$\lim_k \tau_g(p_k) = \inf \tau_g, \quad  \lim_k
\|\grad_g\tau_g\|_g^2(p_k) = 0, 
 \quad \text{and}  \quad  \lim_k(\Delta_g\,\tau_g)(p_k) \ge 0\,.
$$ This gives:
\begin{eqnarray*}
&& 0 \le  nc = n\lim_k (\kappa_{\tilde g} - \tau_g^{-1}\kappa_g)(p_k)\\
& =& -\tau_g^{-2}\,
\Delta_g\tau_g +
\tfrac{1}{4}(n-6)\tau_g^{-3}\,\|\grad_g\tau_g\|_g^2 \} \le 0, 
\end{eqnarray*} thus $c = 0$ and $\kappa_{\tilde g} = \tau_g^{-1}\kappa_g$. The PDE
reduces to
$$\Delta_g \tau_g = - \tau_g^{-1}\tfrac{1}{4}\,(n-6)\cdot  \|\grad_g\tau_g\|_g^2 \ge 0\,.$$
Now we apply  Yau as above.                                                     \end{proof}
\subsection{Trivial Weyl manifolds}
We have the following useful result that characterizes {\it trivial} Weyl manifolds:
\begin{theorem}\label{thm-3.11}
Let $\mathcal{W}=(M,g,\nabla)$ be a Weyl manifold with $H^1(M;\mathbb{R})=0$. The following assertions are equivalent and if any is
satisfied, we say that $\mathcal{W}$ is {\rm trivial}.
\begin{enumerate}
\item $d\phi=0$.
\item $\nabla=\nabla^{g_1}$ for some $g_1$ in the conformal class defined by $g$.
\item $\nabla=\nabla^{g_1}$ for some semi-Riemannian metric $g_1$.
\item $R_P(\nabla)\in\mathfrak{A}$ for every $P\in M$.
\end{enumerate}
\end{theorem}
\begin{proof} Suppose that $d\phi=0$. Since $H^1(M;\mathbb{R})=0$, we can express $\phi=df$ for some function $f$.
Then $\nabla$ is the Levi-Civita connection for the conformally equivalent metric $g_1:=e^{2f}g$. Thus Assertion (1)
implies Assertion (2). Clearly Assertion (2) implies Assertion (3). Since the curvature tensor of the Levi-Civita connection
is algebraic, Assertion (3) implies Assertion (4). Suppose that Assertion (4) holds. We apply Theorem 6 of \cite{GNU10} to
see $d\phi=-\frac1n\Lambda\operatorname{Ric}$. Since the curvature tensor is algebraic, $\Lambda\operatorname{Ric}=0$.
Thus Assertion (4) implies Assertion (1).
\end{proof}
In view of Theorem \ref{thm-3.11}, Proposition \ref{prop-1.3} and Lemma \ref{lem-2.22} offer simple tools
to verify, in terms of curvature decompositions, whether a
Weyl manifold  is trivial.
The following Theorem characterizes trivial Weyl manifolds;
for this we compare 
 gauge invariant global scalar curvatures.
\begin{theorem}
Let $\mathcal{W}$  be a compact Weyl manifold  without boundary,
with positive definite
Weyl metric  $\tilde g$
and associated Riemannian volume form $\omega_{\tilde g} $. Then the total volume of $(N,\tilde g)$ satisfies 
$$\int_N  \omega_{\tilde g}  \le n(n-1)\int_N \kappa_{\tilde g}\, \omega_{\tilde g} \,.$$
Equality holds if and only if $\nabla$ is the Levi-Civita connection of $\tilde g$. 
\end{theorem}
\begin{proof} 
We integrate the second relation in Proposition \ref{prop-3.8}   and apply
Stokes' theorem.  Equality gives $\tilde\phi = 0$ and thus
$\nabla = \nabla^{\tilde g}$.
\end{proof}
\subsection{Einstein-Weyl manifolds}
We recall the following well known definition.
\begin{definition}
\rm A Weyl manifold $(N,g,\nabla)$ is said to be an
{\it Einstein-Weyl manifold} 
if the Ricci tensor $ \operatorname{Ric} = \operatorname{Ric}(\nabla)$ satisfies
\begin{equation}\label{eqn-3.f}
S\operatorname{Ric} = \lambda \cdot g\,.
\end{equation}
\end{definition}
\begin{remark}\rm 
\ \begin{enumerate}\item Obviously the  foregoing  relation is conformally  
invariant,  and  
$$n \cdot \lambda = \tau_g\text{ with }\tau_g= 
\operatorname{Tr}_g\operatorname{Ric}\text{ and }S \operatorname{Ric}
=
\tfrac{1}{n}\; \tau_g \cdot g\,.$$
\item It follows from Theorem \ref{thm-1.1} that 
also $ S \operatorname{Ric}^\star = \tfrac{1}{n}\, \tau_g\cdot g$ \; on 
Einstein-Weyl manifolds.
\item Let the Weyl metric $\tilde g$ be well defined.
Then  $\mathcal{W}$ is Einstein-Weyl if and only if 
$$S \operatorname{Ric} = \tfrac{1}{n} \tau_{\tilde g} \cdot \tilde g =
\tfrac{1}{n} \cdot \tilde g\,.$$
 Thus, in particular, by renormalizing the metric, if $\tau>0$, then we can assume that the Einstein multiple $\lambda$ of
Equation (\ref{eqn-3.f}) satisfies
$\lambda=+1$. Note that in the semi-Riemannian setting, one automatically has $\lambda$ is constant if the dimension is at least 3.
\end{enumerate}\end{remark}
We apply the decomposition results from Section \ref{sect-2.3} and 
immediately get
the following characterizations of Einstein-Weyl manifolds
in terms of components of the decompositions, where
again $R = R(\nabla)$.
\begin{proposition} Let $\mathcal{W}$ be a Weyl manifold. 
Then the following properties are equivalent:
$$\begin{array}{rlrl}
(1)&\mathcal{W}\text{ is Einstein-Weyl.}&
(2)&R = (\alpha_1 + \alpha_6 + \alpha_4  + \alpha_5)(R).\phantom{.......}\\
(3)&R = (\pi_1 +  \pi_3 +   \pi_4 +   \pi_6)(R).&
(4)&\alpha_2(R)= 0.\\
(5)&\pi_2(R)= 0. &
 (6)&\pi_5(R)= 0. \\
 (7)&\operatorname{Ric}(\alpha_2(R)) = 0.&
(8)&\operatorname{Ric}^\star(\alpha_2(R))= 0.\\
(9)&\operatorname{Ric}(\pi_2(R)) = 0.&
(10)&\operatorname{Ric}^\star(\pi_2(R))  = 0. \\
(11)&\operatorname{Ric}^\star(\pi_5(R))= 0.&
(12)&\operatorname{Ric}^\star(\alpha_1(R))= \operatorname{Ric}(R). \\
(13)&  \operatorname{Ric}(\pi_1(R))  = \operatorname{Ric}(R). &
(14)&\operatorname{Ric}^\star(\pi_1(R)) = \operatorname{Ric}(R). \\
(15)&\operatorname{Ric}(\alpha_1(R)) = \operatorname{Ric}(R).\phantom{.......}\end{array}$$
\end{proposition}
\begin{remark}\rm 
\ \begin{enumerate}\item The characterization of Einstein-Weyl manifolds
in Assertion $(5)$ generalizes a  well known characterization 
of Einstein spaces within the class of semi-Riemann manifolds,
using the decomposition of algebraic curvature tensors.
\item Trivially any Ricci-flat Weyl structure is Einstein-Weyl.
\end{enumerate}\end{remark}
\subsection{The conformal and the projective structure of  Weyl
manifolds}
Introducing the concept of Weyl geometry, it was Weyl's intention
to relate the conformal class with the projective structure of 
what we call the Weyl connection $\nabla$.
Concerning the decompositions that we studied and the conformal class,
we recall Section \ref{sect-3.2}. For fixed metric $g$ and 
concerning the projective structure,
according to Section \ref{sect-2.6}  the projective curvature tensor appears in the  $W$-decomposition as
$$p(R) = \bigoplus_4^6 \pi_i(R); $$
to this decomposition there 
corresponds    a gauge invariant  decomposition
of the curvature operator $\mathcal{R}.$ 
Following Section \ref{sect-2.2},  the Higa term can be expressed
as $$H(R) = \pi_3(R) + \pi_4(R)\,.$$
These observations  and Proposition \ref{prop-2.17}
finally lead to the following Theorem  which 
in particular generalizes  a classical result, namely:
{\it A projectively flat semi-Riemannian manifold has constant
sectional curvature.}
\begin{theorem} Let $(N,g,\nabla)$ be a Weyl manifold with $H^1(M;\mathbb{R})=0$.
\begin{enumerate}
\item  If $H(\mathcal{R}) = 0$ then the Weyl manifold is trivial\,.
\item If $\nabla$ is 
projectively flat then the 
curvature tensor satisfies $R = \pi_1(R)$,
i.e. $R$ is of constant curvature type;
in particular, $\mathcal{W}$ is Einstein-Weyl. Moreover, $\Lambda{\operatorname{Ric}} = 0$, thus $(N,g,\nabla)$ is trivial.
\end{enumerate}
\end{theorem}
\section{\bf Length curvature and directional curvature}\label{sect-4}
Let $\mathcal{W} = (N,g,\nabla)$ be a Weyl-manifold.
We use Theorem \ref{thm-1.1} to see:
$$R(x,y,z,w) + R(x,y,w,z) = 2 F(x,y)\, g(z,w),$$
where the operator  $F$ satisfies $F(x,y) := - \tfrac{2}{n} \, (\Lambda \operatorname{Ric})(x,y) = - 2 \,d \phi (x,y);$ here $d$ denotes exterior derivation.
Parts of the following can be found in Section 1 of \cite{H94}.
\begin{definition}\rm
$F$ is called the  {\it length  curvature},
and
$$\mathcal{K}(x,y)z := \mathcal{R}(x,y)z - F(x,y)z$$
is called the    {\it directional curvature} of $\mathcal{W}$. Perlick \cite{P91}  explains this terminology.
\end{definition}
\begin{lemma}
$F$ and $\mathcal{K}$ are gauge invariant.
\end{lemma}
For  $R = R(\nabla) \in \mathfrak{W}$, according to Section \ref{sect-1.4},  we consider its
conjugate curvature tensor $R^\star$. We write $g( \mathcal{K}(x,y)z,w) =:
K(x,y,z,w)$. Recall that the length curvature operator $\mathcal{F}$ was defined by setting $\mathcal{F}(x,y)z:=F(x,y)z$.
Then the foregoing definitions
immediately give:
\begin{lemma}\ 
\begin{enumerate}
\item $R^\star(x,y,w,z) = R(x,y,w,z) - 2 F(x,y)\,g(z,w)$.
\item $\mathcal{R}^\star$ is gauge invariant.
\item $\mathcal{F}= \tfrac{1}{2}\,(\mathcal{R} - \mathcal{R}^\star)$.
\item  $\mathcal{K} = \tfrac{1}{2}\,(\mathcal{R} + \mathcal{R}^\star)$.
\item $g(\mathcal{K}(x,y)z,w) + g(\mathcal{K}(x,y)w,z) = 0.$
\end{enumerate}
\end{lemma}
\begin{corollary}
Adopt the notation given above. We have the relations:
\medbreak
$g(\mathcal{F}(x,y)z,w) = \tfrac{1}{2}(R\, - \, R^\star)(x,y,z,w)
=\tfrac{1}{2} \left(H(A)(x,y,z,w) +  H(A)(x,y,w,z)\right)$,
\medbreak
$g(\mathcal{K}(x,y)z,w) = \tfrac{1}{2}(R\, + \, R^\star)(x,y,z,w) =$
\medbreak\qquad\quad$=(\alpha_1 +  \alpha_2 + \alpha_6)(R)(x,y,z,w) 
+ \tfrac{1}{2} \left(H(A)(x,y,z,w) -  H(A)(x,y,w,z)\right)$.
\end{corollary}
The relations in (3) and (4)  in the foregoing Lemma  show a ``symmetry'' in the definition
of $\mathcal{K}$ and $\mathcal{F}$ and thus justify consideration of the conjugate curvature operator $\mathcal{R}^\star$
(conjugate curvature tensor $R^\star$, respectively) in Weyl geometry. Moreover, the foregoing Corollary clarifies 
the role of the Higa term for $\mathcal{K}$ and $\mathcal{F}$.
From  Proposition \ref{prop-1.3} recall that, in general,  $\mathcal{K}$ does not satisfy the Bianchi identity,   
 it is not algebraic.
 The following Lemma summarizes simple characterizations of trivial Weyl manifolds.
\begin{lemma}
We have the equivalences:
$$\begin{array}{llllll} 
(1)&F= 0.&(2)&H = 0.&(3)&R= R^\star.\\
(4)&R= K.&(5)&\Lambda\operatorname{Ric}= 0.&(6)&\alpha_4(R)=0.\\ 
(7)&\alpha_5(R)=0.&
(8)&\pi_3(R)=0.&
(9)&\pi_4(R)=0.\\
(10)&\mathcal{R}^\star \text{ satisfies (\ref{eqn-1.c})}.&
(11)& \mathcal{F}\text{ satisfies (\ref{eqn-1.c})}.&
(12)&\mathcal{K}\text{ satisfies  (\ref{eqn-1.c})}.\phantom{...}\\
(13)&R\in\mathfrak{A}.& 
(14)&R^\star\in\mathfrak{A}.\\\end{array}$$
\begin{enumerate}
\item[(15)]
There exists a metric $g_1$ in the conformal
class of $\mathcal{W}$ such that its Levi-Civita connection
$\nabla^{g_1}$ coincides with  the Weyl connection.
\end{enumerate}
\end{lemma}
With the results from Section 2 in \cite{GNU09} we get:
\begin{corollary} The Ricci tensors $\operatorname{Ric} = \operatorname{Ric}(R)$
and  $\operatorname{Ric}^\star = \operatorname{Ric}^\star(R)  $ satisfy:
$$\begin{array}{llll}
(1)& \operatorname{Ric}(\mathcal{K})=\tfrac12(\operatorname{Ric} +
\operatorname{Ric}^\star).&
(2)&S \operatorname{Ric}(\mathcal{K}) =S \operatorname{Ric}.\\
(3)& \Lambda \operatorname{Ric}(\mathcal{K}) = \tfrac{1}{n}\,(n-2) \Lambda \operatorname{Ric}.\phantom{...............}&
(4)&\operatorname{Ric}(\mathcal{\mathcal{F}}) = \tfrac{1}{n} \Lambda \operatorname{Ric}.\phantom{................}\\
(5)& S \operatorname{Ric}(\mathcal{F})=0.\end{array}$$
\end{corollary}
\begin{lemma} Let $\nabla$ be projectively flat. Then $F = 0.$
\end{lemma}
\begin{proof} Projective flatness implies $\Lambda \operatorname{Ric} = 0,$
thus $F = 0.$
\end{proof}
\begin{proposition} Let $\mathcal{W}$ be an $n$-dimensional  Weyl manifold
such that the length  curvature is non-zero
at least at one point. The following assertions are equivalent:
 $$\begin{array}{llll}
(1)& n = 4.\phantom{.............}\qquad\qquad
&(2)&\Lambda \operatorname{Ric}^\star = 0.\qquad\qquad\qquad\qquad\qquad\qquad\end{array}$$
\end{proposition}
\begin{remark}\rm
Let $n\ne4$. Then $\Lambda(\operatorname{Ric})=0$ if and only if $\Lambda(\operatorname{Ric}^\star)=0$.
\end{remark}
\section*{Acknowledgments} 
Research of P. Gilkey partially supported by DFG PI 158/4-6 (Germany) and by project MTM2009-07756 (Spain). Research of
S. Nik\v cevi\'c partially supported by a research grant of the TU Berlin, by project MTM2009-07756 (Spain), and by 144032
(Serbia). Research of U. Simon partially supported by DFG PI 158/4-6 (Germany). We thank V. Perlick for hints.

\end{document}